\documentclass[oneside,english]{amsart}
\usepackage[T1]{fontenc}
\usepackage[latin9]{inputenc}
\usepackage{geometry}
\usepackage{amsthm}
\usepackage{amssymb}

\makeatletter
\numberwithin{equation}{section}
\numberwithin{figure}{section}
  \theoremstyle{plain}
  \newtheorem*{thm*}{\protect\theoremname}
\theoremstyle{plain}
\newtheorem{thm}{\protect\theoremname}
  \theoremstyle{plain}
  \newtheorem{prop}[thm]{\protect\propositionname}
  \theoremstyle{definition}
  \newtheorem{example}[thm]{\protect\examplename}
  \theoremstyle{remark}
  \newtheorem{rem}[thm]{\protect\remarkname}
  \theoremstyle{plain}
  \newtheorem{lem}[thm]{\protect\lemmaname}
  \theoremstyle{plain}
  \newtheorem{cor}[thm]{\protect\corollaryname}

 \usepackage{xypic}

\makeatother

\usepackage{babel}
  \providecommand{\corollaryname}{Corollary}
  \providecommand{\examplename}{Example}
  \providecommand{\lemmaname}{Lemma}
  \providecommand{\propositionname}{Proposition}
  \providecommand{\remarkname}{Remark}
  \providecommand{\theoremname}{Theorem}
\providecommand{\theoremname}{Theorem}

\begin{document}

\title{Families of exotic affine $3$-spheres}

\author{Adrien Dubouloz}

\address{IMB UMR5584, CNRS, Univ. Bourgogne Franche-Comté, F-21000 Dijon,
France.}

\email{adrien.dubouloz@u-bourgogne.fr}

\thanks{This research was partialy funded by ANR Grant \textquotedbl{}BirPol\textquotedbl{}
ANR-11-JS01-004-01. }

\subjclass[2000]{14R05, 14R25, 14J10.}
\begin{abstract}
We construct algebraic families of exotic affine $3$-spheres, that
is, smooth affine threefolds diffeomorphic to a non-degenerate smooth
complex affine quadric of dimension $3$ but non algebraically isomorphic
to it. We show in particular that for every smooth topologically contractible
affine surface $S$ with trivial automorphism group, there exists
a canonical smooth family of pairwise non isomorphic exotic affine
$3$-spheres parametrized by the closed points of $S$. 
\end{abstract}

\maketitle

\section*{Introduction}

An\emph{ exotic affine $n$-sphere }is a smooth complex affine variety
diffeomorphic to the standard affine algebraic $n$-sphere $\mathbb{S}^{n}=\left\{ x_{1}^{2}+\cdots+x_{n+1}^{2}=1\right\} $
in $\mathbb{A}_{\mathbb{C}}^{n+1}$ but not algebraically isomorphic
to it. In dimension $n=1$ or $2$, the algebraic structure on a smooth
affine variety diffeomorphic to $\mathbb{S}^{n}$ is actually uniquely
determined by its topology: $\mathbb{S}^{1}\simeq\mathbb{A}_{\mathbb{C}}^{1}\setminus\left\{ 0\right\} $
is the unique smooth affine curve $C$ with $H_{1}\left(C;\mathbb{Z}\right)\simeq\mathbb{Z}$,
and a smooth affine surface $S$ is algebraically isomorphic to $\mathbb{S}^{2}$
if and only if it has the same homology type and the same homotopy
type at infinity as $\mathbb{S}^{2}$ (see \cite[Theorem 3.3]{DuFin}).
Examples of smooth affine threefolds with the same homology type as
$\mathbb{S}^{3}$ were first constructed by D. Finston and S. Maubach
in the form of total spaces of certain locally trivial $\mathbb{A}^{1}$-bundles
over the smooth locus $S_{p,q,r}^{*}=S_{p,q,r}\setminus\{(0,0,0)\}$
of a Brieskorn surface 
\[
S_{p,q,r}=\left\{ x^{p}+y^{q}+z^{r}=0\right\} \subset\mathbb{A}^{3},\quad1/p+1/q+1/r<1.
\]
These threefolds are all diffeomorphic to each others, admitting the
corresponding Brieskorn homology sphere $\Sigma\left(p,q,r\right)$
as a strong deformation retract. The main result of \cite{Finston2008}
asserts that their isomorphy types as abstract varieties are uniquely
determined by their isomorphy classes as $\mathbb{A}^{1}$-bundles
over $S_{p,q,r}^{*}$, up to composition by automorphisms of $S_{p,q,r}^{*}$.
In particular, their construction gives rise to discrete families
of exotic ``homology'' affine $3$-spheres (see Example \ref{exa:FinMau }).

Natural candidates of exotic affine $3$-spheres are the total spaces
of nontrivial $\mathbb{A}^{1}$-bundles over the punctured affine
plane $\mathbb{A}_{*}^{2}$. Indeed, noting that $\mathbb{S}^{3}\simeq{\rm SL}_{2}\left(\mathbb{C}\right)=\left\{ xv-yu=1\right\} \subset\mathbb{A}_{*}^{2}\times\mathbb{A}^{2}$,
where $\mathbb{A}_{*}^{2}={\rm Spec}\left(\mathbb{C}\left[x,y\right]\right)\setminus\{(0,0)\}$,
one checks more generally that all the varieties

\[
X_{m,n}=\left\{ x^{m}v-y^{n}u=1\right\} \subset\mathbb{A}_{*}^{2}\times\mathbb{A}^{2},\quad m+n\geq2,
\]
are diffeomorphic to the trivial $\mathbb{R}^{2}$-bundle $\mathbb{R}^{2}\times\mathbb{A}_{*}^{2}$
over $\mathbb{A}_{*}^{2}$ when equipped with the euclidean topology.
It was established in \cite{DuFin} that for $m+n>2$, $X_{m,n}$
is indeed an exotic affine $3$-sphere. But the classification of
all threefolds $X_{m,n}$, $m+n>2$, up to isomorphism remains highly
elusive, in particular, the techniques developed in \emph{loc. cit.}
do not allow to distinguish the $X_{m,n}$, $m+n>2$, from each others.
As a consequence, all these threefolds could very well be all isomorphic,
forming thus a unique class of exotic affine $3$-spheres. 

In this article, we build up on these ideas to construct infinitely
many pairwise non isomorphic new exotic affine $3$-spheres, arising
as total spaces of $\mathbb{A}^{1}$-bundles over $1$-punctured smooth
topologically contractible algebraic surfaces $S$. In particular,
we obtain the following: 
\begin{thm*}
For every smooth topologically contractible complex algebraic surface
$S$, there exists a canonical smooth family $h:\mathcal{V}\rightarrow S$
of affine threefolds with the following properties: 

a) The closed fibers are all diffeomorphic to $\mathbb{S}^{3}$,

b) The fibers of $h:\mathcal{V}\rightarrow S$ over two closed points
$p$ and $p'$ of $S$ are isomorphic if and only if $p$ and $p'$
belong to the same orbit of the action of $\mathrm{Aut}(S)$ on $S$. 
\end{thm*}
If $S=\mathbb{A}^{2}$, then the family $h:\mathcal{V}\rightarrow\mathbb{A}^{2}$
is actually isomorphic to the trivial one $\mathbb{S}^{3}\times\mathbb{A}^{2}$.
But in contrast, if $S$ is a general smooth topologically contractible
complex algebraic surface $S$ of non-negative Kodaira dimension $\kappa(S)$,
$\mathrm{Aut}(S)$ is trivial and the closed fibers of $h:\mathcal{V}\rightarrow S$
are pairwise non isomorphic exotic affine $3$-spheres. 

\section{Preliminaries}

Here we review a construction of Zariski locally trivial $\mathbb{A}^{1}$-bundles
over the complement of certain closed subschemes of pure codimension
$2$ in a given scheme that we will use repeatedly in the article.
We establish a criterion for isomorphism of total spaces of certain
of these bundles. 

\subsection{\label{subsec:Bundle-Const}Codimension $2$ subschemes and locally
trivial $\mathbb{A}^{1}$-bundles }

Suppose given a scheme $Y$ and closed subscheme $Z\subset Y$ of
pure codimension $2$ with ideal sheaf $\mathcal{I}_{Z}\subset\mathcal{O}_{Y}$
for which we have a locally free resolution \begin{eqnarray}\label{eq:Seq1} \xymatrix{ 0 \ar[r] & \mathcal{L} \ar[r]^{a} & \mathcal{F} \ar[r]^{b} & \mathcal{I}_Z \ar[r] & 0,}\end{eqnarray} 
where $\mathcal{L}$ and $\mathcal{F}$ are locally free sheaves on
$Y$, of rank $1$ and $2$ respectively. In particular, $Z$ is a
local complete intersection in $Y$. We let $\overline{q}:X=\mathbb{P}(\mathcal{F})\rightarrow Y$
be the relative $\mathrm{Proj}$ of the symmetric algebra $\mathrm{Sym}_{Y}^{\cdot}\mathcal{F}$
of $\mathcal{F}$ and we let $H\subset X$ be the closed subscheme
determined by the surjection $b:\mathcal{F}\rightarrow\mathcal{I}_{Z}$.
We let $V_{Z}=X\setminus H$ and we let $p:L=\mathrm{Spec}_{Y\setminus Z}(\mathrm{Sym}^{\cdot}\mathcal{L}^{-1}|_{Y\setminus Z})\rightarrow Y\setminus Z$
be the restriction to $Y\setminus Z$ of the total space of the invertible
sheaf $\mathcal{L}$.
\begin{prop}
\label{prop:Bundle-construction}The scheme $V_{Z}$ is affine over
$Y$ and $\overline{q}|_{V_{Z}}$ factors through an $L$-torsor $\rho_{Z}:V_{Z}\rightarrow Y\setminus Z$
whose isomorphy class in $H^{1}(Y\setminus Z,L)\simeq\mathrm{Ext}_{Y\setminus Z}^{1}(\mathcal{O}_{Y\setminus Z},\mathcal{L})$
coincide with that of the restriction \[\xymatrix{0 \ar[r] & \mathcal{L}|_{Y\setminus Z} \ar[r]^{a} & \mathcal{F}|_{Y\setminus Z} \ar[r]^-{b} & \mathcal{I}_Z|_{Y\setminus Z}\simeq \mathcal{O}_{Y\setminus Z} \ar[r] & 0}\]of
the exact sequence $($\ref{eq:Seq1}$)$ to $Y\setminus Z$. . 
\end{prop}
\begin{proof}
By definition, $H$ is equal to the closed subscheme $\mathrm{Proj}_{Y}(\mathrm{Sym}^{\cdot}\mathcal{I}_{Z})$
of $X$, and since $Z$ is a local complete intersection of codimension
$2$ in $Y$, $\mathrm{Sym}_{Y}^{\cdot}\mathcal{I}_{Z}$ is locally
isomorphic as a sheaf of graded $\mathcal{O}_{Y}$-algebra to the
Rees $\mathcal{O}_{Y}$-algebra $\bigoplus_{n\geq0}\mathcal{I}_{Z}^{n}$
of $\mathcal{I}_{Z}$. So locally over $Y$, $\overline{q}|_{H}:H\rightarrow Y$
is isomorphic to the blow-up of $Y$ along $Z$, implying in particular
that $\overline{q}^{-1}(Z)$ is fully contained in $H$ while $H|_{X\setminus\overline{q}^{-1}(Z)}$
is a section of the restriction of the locally trivial $\mathbb{P}^{1}$-bundle
$\overline{q}:X\rightarrow Y$ over $Y\setminus Z$. As a consequence,
the image of $q$ is equal to $Y\setminus Z$ and the corestriction
$\rho_{Z}:V_{Z}\rightarrow Y\setminus Z$ of $\overline{q}|_{V_{Z}}$
to its image is a locally trivial $\mathbb{A}^{1}$-bundle. Since
$H$ is relatively ample over $Y$, $V_{Z}=X\setminus H$ is affine
over $Y$, and  the fact that $\rho_{Z}:V_{Z}\rightarrow Y\setminus Z$
is actually an $L$-torsor with the given isomorphy class follows
directly from the construction. 
\end{proof}
\begin{example}
\label{exa:CompleteInter}In the case where $Z$ is a global complete
intersection in $Y$, the above construction applied to a free resolution
\[\xymatrix{ 0 \ar[r] & \mathcal{O}_Y \ar[rr]^{^t(-g,f)}  & & \mathcal{O}_Y^{\oplus 2} \ar[rr]^-{(f,g)} & & \mathcal{I}_Z \ar[r] & 0} \]
for some global regular functions $f,g\in\Gamma(Y,\mathcal{O}_{Y})$
such that $Z=V(f,g)$ yields a scheme $X$ is isomorphic to $Y\times\mathbb{P}_{[u:v]}^{1}$,
for which $H$ is equal to the closed subscheme $\left\{ fv-gu=0\right\} $.
The morphism 
\[
V_{Z}=X\setminus H\rightarrow Y\times\mathrm{Spec}(\mathbb{C}[U,V]),\quad(y,[u:v])\mapsto(y,\frac{u}{fv-gu},\frac{v}{fv-gu})
\]
then induces an isomorphism between $V_{Z}$ and the closed sub-subscheme
of $Y\times\mathrm{Spec(\mathbb{C}[U,V])}$ with equation $fV-gU=1$.
Note that if \[\xymatrix{ 0 \ar[r] & \mathcal{O}_Y \ar[rr]^{^t(-g',f')}  & & \mathcal{O}_Y^{\oplus 2} \ar[rr]^-{(f',g')} & & \mathcal{I}_Z \ar[r] & 0} \]is
another free resolution, then there exists $\varphi=\left(\begin{array}{cc}
a & b\\
c & d
\end{array}\right)\in\mathrm{GL}_{2}(\Gamma(Y,\mathcal{O}_{Y}))$ such that $f'=af+bg$ and $g'=cf+dg$. This element $\varphi$ determines
in turn an $Y$-automorphism $(U,V)\mapsto(-cU+aV,dU-bV)$ of $Y\times\mathrm{Spec}(\mathbb{C}[U,V])$
which maps $\{f'V-g'U=1\}$ isomorphically onto $\{fV-gU=1\}$.
\end{example}

\subsection{A criterion for isomorphism }

The following proposition provides a criterion to decide when the
total spaces of certain torsors $\rho_{Z}:V_{Z}\rightarrow Y\setminus Z$
and $\rho_{Z'}:V_{Z'}\rightarrow Y'\setminus Z'$ as above are isomorphic
as abstract algebraic varieties. 
\begin{prop}
\label{prop:Isocrit}Let $(Y,Z)$ and $(Y',Z')$ be pairs consisting
of a normal affine variety and a scheme theoretic global complete
intersection of pure codimension $2$, and let $\rho_{Z}:V_{Z}\rightarrow Y\setminus Z$
and $\rho_{Z'}:V_{Z'}\rightarrow Y'\setminus Z'$ be the corresponding
varieties. If either $Y\setminus Z$ or $Y'\setminus Z$ has non-negative
Kodaira dimension, then $V_{Z}$ and $V_{Z'}$ are isomorphic as abstract
varieties if and only if the pairs $(Y,Z)$ and $(Y',Z')$ are isomorphic. 
\end{prop}
\begin{proof}
Given an isomorphism $\varphi:Y\rightarrow Y'$ mapping $Z$ onto
$Z'$, it follows from the construction of $V_{Z}$ and $V_{Z'}$
that $V_{Z}$ and $V_{Z'}\times_{S'}S\simeq V_{Z'}$ are isomorphic
as schemes over $Y$. Conversely, if either $\kappa(Y\setminus Z)$
or $\kappa(Y'\setminus Z')$ is non-negative, then it follows from
a straightforward adaptation of the proof of the Iitaka-Fujita Strong
Cancellation Theorem \cite{IiFu77} to the case of locally trivial
$\mathbb{A}^{1}$-bundles, that every isomorphism $\Psi:V_{Z}\rightarrow V_{Z'}$
descends to a unique isomorphism $\psi:Y\setminus Z\rightarrow Y'\setminus Z'$
for which the following diagram is commutative \[\xymatrix{ V_p \ar[r]^{\Psi} \ar[d]_{\rho_Z} & V_{p'} \ar[d]^{\rho_{Z'}} \\ Y\setminus Z \ar[r]^{\psi}& Y'\setminus Z'. }\]
Since $Y$ and $Y'$ are affine and normal, and $Z$ and $Z'$ have
pure codimension $2$ in $Y$ and $Y'$ respectively, $\psi$ uniquely
extends to an isomorphism $\varphi:Y\rightarrow Y'$ such that $\varphi(Z_{\mathrm{red}})=Z'_{\mathrm{red}}$.
Let $Z=V(f,g)$, $Z'=V(f',g')$ and let $Z''=\varphi^{-1}(Z')=V(f'',g'')$
where $f''=f'\circ\varphi$ and $g''=g'\circ\varphi$. The isomorphism
$\Psi$ factors through an isomorphism $\tilde{\Psi}:V_{Z}\rightarrow V_{Z''}\simeq V_{Z'}\times_{Y'}Y$
of locally trivial $\mathbb{A}^{1}$-bundles over $Y\setminus Z$.
Now $V_{Z}$ and $V_{Z''}$ are isomorphic as locally trivial $\mathbb{A}^{1}$-bundles
over $Y_{*}=Y\setminus Z=Y\setminus Z''$ if and only if there exists
an isomorphism of extensions \[\xymatrix{0 \ar[r] & \mathcal{O}_{Y_*} \ar@{=}[d] \ar[rr]^{^t(-g,f)} & & \mathcal{O}_{Y_*}^{\oplus 2} \ar[d]^{\wr}_{\varphi} \ar[rr]^-{(f,g)} & & \mathcal{I}_Z|_{Y_*}\simeq \mathcal{O}_{Y_*} \ar[d]^{\wr}_{\times \lambda} \ar[r] & 0 \\ 0 \ar[r] & \mathcal{O}_{Y_*} \ar[rr]^{^t(-g'',f'')} & & \mathcal{O}_{Y_*}^{\oplus 2} \ar[rr]^-{(f'',g'')} & & \mathcal{I}_{Z''}|_{Y_*}\simeq \mathcal{O}_{Y_*} \ar[r] & 0,}\]where
the isomorphism on the middle is given by an element $\varphi=\left(\begin{array}{cc}
a & b\\
c & d
\end{array}\right)\in\mathrm{GL}_{2}(\Gamma(Y_{*},\mathcal{O}_{Y_{*}}))=\mathrm{GL}_{2}(\Gamma(Y,\mathcal{O}_{Y}))$ and the isomorphism on the right-hand side is the multiplication
by an element $\lambda\in\Gamma(Y_{*},\mathcal{O}_{Y_{*}}^{*})=\Gamma(Y,\mathcal{O}_{Y}^{*})$.
So $\mathcal{I}_{Z''}=((af+bg),(cf+dg))=(\lambda f,\lambda g)=\mathcal{I}_{Z}$,
which implies that $\varphi(Z)=Z'$ as desired. 
\end{proof}
\begin{example}
\label{exa:FinMau }(See \cite[Example 1]{Finston2008}) Since the
smooth locus $S_{p,q,r}^{*}=S_{p,q,r}\setminus\{(0,0,0)\}$ of a Brieskorn
surface $S_{p,q,r}=\left\{ x^{p}+y^{q}+z^{r}=0\right\} $ in $\mathbb{A}^{3}$,
where $p,q,r$ are pairwise relatively prime and $1/p+1/q+1/r<1$,
has Kodaira dimension equal to $1$ \cite{Ii77-2}, it follows from
Proposition \ref{prop:Isocrit} that the total spaces of the $\mathbb{A}^{1}$-bundles
\[
\rho_{m,n}:X_{m,n}=\left\{ x^{m}V-y^{n}U=1\right\} \rightarrow S_{p,q,r}^{*},\quad m,n\geq1
\]
associated to the closed subschemes $Z_{m,n}\subset S_{p,q,r}$ with
defining ideals $I_{m,n}=\left(x^{m},y^{n}\right)$, $m,n\geq1$ are
isomorphic if and only if the corresponding pairs $(S_{p,q,r},Z_{m,n})$
are isomorphic. Since $\mathrm{Aut}(S_{p,q,r})$ is isomorphic to
the multiplicative group $\mathbb{G}_{m}$, acting linearly on $S_{p,q,r}$
by $\lambda\cdot(x,y,z)=(\lambda^{qr}x,\lambda^{pr}y,\lambda^{qr}z)$,
two such pairs $(S_{p,q,r},Z_{m,n})$ and $(S_{p,q,r},Z_{m',n'})$
are isomorphic if and only if $(m,n)=(m',n')$, and so the varieties
$X_{m,n}$ are pairwise non isomorphic. 
\end{example}
\begin{rem}
In Proposition \ref{prop:Isocrit}, the hypothesis on the Kodaira
dimension is crucial. For instance, let $\mathbb{A}^{2}=\mathrm{Spec}(\mathbb{C}[x,y])$
and consider the $\mathbb{A}^{1}$-bundles 
\[
\rho_{m,n}:X_{m,n}=\left\{ x^{m}V-y^{n}U=1\right\} \rightarrow\mathbb{A}^{2}\setminus\{(0,0)\}
\]
associated to the closed subschemes $Z_{m,n}\subset\mathbb{A}^{2}$
with defining ideals $I_{m,n}=\left(x^{m},y^{n}\right)$, $m,n\geq1$.
It is not difficult to check that two pairs $(\mathbb{A}^{2},Z_{m,n})$
and $(\mathbb{A}^{2},Z_{m',n'})$ are isomorphic if and only if $\{m,n\}=\{m',n'\}$
(see e.g. \cite[Proposition 2.2]{DuFin}). On the other hand, it was
established in \cite[Theorem 2.3]{DuFin} that if $m+n=m'+n'$ then
$X_{m,n}$ and $X_{m',n'}$ are isomorphic as abstract varieties. 
\end{rem}

\section{algebraic families of exotic $3$-spheres over punctured contractible
surfaces}

In this section, we construct algebraic families of affine $3$-spheres
in the form of locally trivial $\mathbb{A}^{1}$-bundles over $1$-punctured
topologically contractible affine surfaces. Recall that a topological
manifold is called contractible if it has the homotopy type of a point.
A smooth complex algebraic variety is called contractible if it is
so when equipped with its underlying structure of real topological
manifold. 

\subsection{Basic properties of contractible surfaces}

Smooth contractible algebraic surfaces have been intensively studied
during the last decades. We just recall a few basic facts about these,
referring the reader to \cite[$\S$ 2]{Za00} and \cite[Chapter 3]{MiyBook}
and the references therein for a complete overview. First of all,
these surfaces are all affine and rational. They are partially classified
in terms of their (logarithmic) Kodaira dimension $\kappa$: $\mathbb{A}^{2}$
is the only such surface of negative Kodaira dimension, there is no
contractible surface of Kodaira dimension $0$, contractible surfaces
of Kodaira dimension $1$ are fully classified in terms of the structure
of their smooth completions. So far, there is no classification of
contractible surfaces of Kodaira dimension $2$, but many families
of examples have been constructed \cite{tD90,GuMi87,MS91,Ram71}. 

The Picard group of a smooth contractible surface $S$ is trivial
\cite[Lemma 4.2.1]{MiyBook}, and more generally, the affineness and
the rationality of $S$ combined with a result of Murthy \cite{Mur69}
imply that every locally free sheaf $\mathcal{E}$ of rank $r\geq2$
on $S$ is free. Of particular interest for us is the following consequence
of these facts: 
\begin{lem}
\label{lem:Contract-Complete-Inter}In a smooth contractible affine
surface $S$, every closed subscheme $Z\subset S$ of pure dimension
which is a scheme theoretic local complete intersection is a scheme
theoretic global complete intersection.

In particular, every closed point $p$ on a smooth contractible surface
$S$ is a scheme theoretic global complete intersection. 
\end{lem}
\begin{proof}
Since the Picard group $\mathrm{Pic}(S)$ is trivial, every scheme
theoretic local complete intersection of pure codimension $1$ is
a global complete intersection. If $Z$ has pure codimension $2$,
hence pure dimension $0$, then its normal sheaf $\mathcal{N}_{Z/S}$
is locally free of rank $2$, hence isomorphic to the trivial sheaf
$\mathcal{O}_{Z}^{\oplus2}$. So $\Lambda^{2}\mathcal{N}_{Z/S}\simeq\mathcal{O}_{Z}$,
and it follows from the Serre construction \cite{Se61} (see also
\cite[Theorem 1.1]{Ar07}) that there exists an exact sequence $0\rightarrow\mathcal{O}_{S}\rightarrow\mathcal{E}\rightarrow\mathcal{I}_{Z}\rightarrow0$,
where $\mathcal{E}$ is locally free of rank $2$, hence free of rank
$2$, providing an isomorphism $Z\simeq V(f,g)$ for some regular
function $f,g\in\Gamma(S,\mathcal{O}_{S})$. 
\end{proof}
We also record the following result concerning automorphism groups
of smooth contractible surfaces:
\begin{prop}
\label{prop:Auto-Contract} The automorphism group of a contractible
surface $S$ different from $\mathbb{A}^{2}$ is either trivial if
$\kappa(S)=1$ or finite if $\kappa(S)=2$. 
\end{prop}
\begin{proof}
The case where $\kappa(S)=2$ follows from the more general fact that
any variety of maximal Kodaira dimension has finite automorphism group
(see e.g. \cite[Theorem 6]{Ii77}). For the case $\kappa(S)=1$, see
\cite{Pe89}. 
\end{proof}
\begin{rem}
It is established in \cite{Pe89} that certain smooth contractible
surfaces $S$ with $\kappa(S)=2$ such as the Ramanujam surface \cite{Ram71}
and the Gurjar-Miyanishi surfaces \cite{GuMi87} have a trivial automorphism
group. But there exists examples of smooth contractible surfaces of
log-general type with nontrivial finite automorphism groups (see e.g.
\cite{MS91,tD90}). 
\end{rem}

\subsection{A construction of affine $3$-spheres}

Let $S$ be a smooth contractible surface and let $Z=V(f,g)\subset S$
be a global scheme theoretic complete intersection whose support is
a closed point $p$ of $S$. By applying the construction described
in \S \ref{subsec:Bundle-Const} to the corresponding resolution
\[\xymatrix{ 0 \ar[r] & \mathcal{O}_S \ar[rr]^{^t(-g,f)}  & & \mathcal{O}_S^{\oplus 2} \ar[rr]^-{(f,g)} & & \mathcal{I}_Z \ar[r] & 0} \]of
the ideal $\mathcal{I}_{Z}\subset\mathcal{O}_{S}$ of $Z$, we obtain
a smooth affine threefold 
\[
V_{Z}\simeq\left\{ fV-gU=1\right\} \subset S\times\mathbb{A}^{2}
\]
equipped with the structure of a locally trivial $\mathbb{G}_{a}$-bundle
$\rho_{Z}=\mathrm{pr}_{S}:V_{Z}\rightarrow S\setminus Z=S\setminus\{p\}$
for the fixed point free $\mathbb{G}_{a,S\setminus\{p\}}$-action
defined by $t\cdot(s,U,V)=(s,U+f(s)t,V+g(s)t)$. 
\begin{example}
Since by Lemma \ref{lem:Contract-Complete-Inter} every closed point
$p$ on a smooth contractible surface $S$ is a global scheme theoretic
complete intersection $p=V(f,g)$ for some $f,g\in\Gamma(S,\mathcal{O}_{S})$,
we obtain in particular for every such point $p$ a smooth affine
threefold 
\[
\rho_{p}:V_{p}(S)\simeq\left\{ fV-gU=1\right\} \rightarrow S\setminus\{p\},
\]
whose isomorphy class as a scheme over $S\setminus\{p\}$ is independent
on the choice of the two generators of the ideal $I_{p}$ of $p$
in $\Gamma(S,\mathcal{O}_{S})$ (see Example \ref{exa:CompleteInter}). 
\end{example}
The variety $V_{o}(\mathbb{A}^{2})$ corresponding to the origin $o=(0,0)$
in $\mathbb{A}^{2}=\mathrm{Spec}(\mathbb{C}[x,y])$ is isomorphic
to $\mathrm{SL}_{2}(\mathbb{C})=\left\{ xV-yU=1\right\} $, hence
to the standard affine $3$-sphere $\mathbb{S}^{3}$. In particular,
it is diffeomorphic to $\mathbb{S}^{3}$. More generally, we have
the following result:
\begin{prop}
\label{prop:Diffeo-Sphere}For every pair $(S,Z)$ consisting of a
smooth contractible surface $S$ and a global scheme theoretic complete
intersection $Z\subset S$ supported at a closed point $p$ of $S$,
the smooth affine threefold $V_{Z}$ is diffeomorphic to the standard
affine $3$-sphere $\mathbb{S}^{3}=\left\{ x_{1}^{2}+x_{2}^{2}+x_{3}^{2}+x_{4}^{2}=1\right\} $
in $\mathbb{A}^{4}$. 
\end{prop}
\begin{proof}
Since the sheaf $\mathcal{C}^{\infty}(S\setminus\{p\},\mathbb{C})$
of complex valued $\mathcal{C}^{\infty}$-functions on $S\setminus\{p\}$
is soft (see e.g. \cite[Theorem 5, p.25]{Gr79} ), every algebraic
\v{C}ech $1$-cocycle representing the class of the $\mathbb{G}_{a,S\setminus\{p\}}$-bundle
$\rho_{Z}:V_{Z}\rightarrow S\setminus Z=S\setminus\{p\}$ in $H^{1}(S\setminus\{p\},\mathcal{O}_{S\setminus\{p\}})$
on some Zariski open covering of $S\setminus\{p\}$ is a coboundary
when considered as a $1$-cocycle with values in $\mathcal{C}^{\infty}(S\setminus\{p\},\mathbb{C})$.
So $\rho_{Z}:V_{Z}\rightarrow S\setminus\{p\}$ is diffeomorphic to
the trivial $\mathbb{R}^{2}$-bundle over $S\setminus\{p\}$, and
hence $V_{Z}$ is diffeomorphic to the complement in the trivial $\mathbb{R}^{2}$-bundle
$\mathrm{pr}_{1}:S\times\mathbb{R}^{2}\rightarrow S$ of $\mathrm{pr}_{1}^{-1}(p)\simeq\mathbb{R}^{2}$.
Since $S$ is contractible, $S\times\mathbb{R}^{2}$ is diffeomorphic
to $\mathbb{R}^{6}$ by virtue of \cite{McMZee62}. The assertion
then follows from the fact that any two smooth closed embeddings of
$\mathbb{R}^{2}$ into $\mathbb{R}^{6}$ are properly homotopic, and
that in this range of dimensions any two properly homotopic smooth
closed embeddings are ambiently isotopic \cite[Chapter 8]{Hi76}. 
\end{proof}
In particular, we obtain the following:
\begin{cor}
\label{cor:Exotic-Sphere}For every pair $(S,Z)$ consisting of a
smooth contractible surface $S$ of non-negative Kodaira dimension
and a global scheme theoretic complete intersection $Z\subset S$
supported at a point, the corresponding threefold $V_{Z}$ is an exotic
affine $3$-sphere. 
\end{cor}
\begin{proof}
This follows from Proposition \ref{prop:Isocrit} applied to the pairs
$(Y,Z)=(S,Z)$ and $(Y',Z')=(\mathbb{A}^{2},\{(0,0\})$. 
\end{proof}
\begin{cor}
Let $S$ and $S'$ be non-isomorphic smooth contractible surface of
non-negative Kodaira dimension. Then for every pair of closed points
$p\in S$ and $p'\in S'$, the threefolds $V_{p}(S)$ and $V_{p'}(S')$
are non-isomorphic exotic affine $3$-spheres. 
\end{cor}
\begin{proof}
This follows again from Proposition \ref{prop:Isocrit} applied to
the pairs $(S,p)$ and $(S',p')$. 
\end{proof}
\noindent Applying Proposition \ref{prop:Isocrit} to pairs of closed
points on a fixed smooth contractible surface, we obtain:
\begin{cor}
\label{cor:Contract-Iso-Crit}Let $S$ be a smooth contractible surface
of non-negative Kodaira dimension and let $p,p'$ be closed point
on $S$. Then $V_{p}(S)$ and $V_{p'}(S)$ are isomorphic as abstract
algebraic varieties if and only if $p$ and $p'$ belong to the same
orbit of the action of $\mathrm{Aut}(S)$ on $S$. 
\end{cor}

\subsection{Algebraic families of exotic affine $3$-spheres}

Given a smooth contractible surface $S$, we let $\Delta\subset S\times S$
be the diagonal, we let $T=S\times S\setminus\Delta$ and we denote
by $\pi:T\rightarrow S$ the restriction of the first projection.
Since the normal sheaf $\mathcal{N}_{\Delta/S\times S}\simeq\mathcal{T}_{S}$
of $\Delta\simeq S$ in $S\times S$ is trivial, it follows from the
Serre construction \cite[Theorem 1.1]{Ar07} that the ideal sheaf
$\mathcal{I}_{\Delta}$ admits a locally free resolution of the form
\begin{eqnarray} \label{eq:Second-Seq}\xymatrix{0 \ar[r] & \mathcal{O}_{S\times S} \ar[r] & \mathcal{F} \ar[r] & \mathcal{I}_{\Delta} \ar[r] & 0}\end{eqnarray}for
some locally free sheaf $\mathcal{F}$ of rank $2$ on $S\times S$,
uniquely determined up to isomorphism. Furthermore, the above extension
can be chosen in such a way that its class in $\mathrm{Ext}_{S\times S}^{1}(\mathcal{I}_{\Delta},\mathcal{O}_{S\times S})$
coincides via the isomorphism $\mathrm{Ext}_{S\times S}^{1}(\mathcal{I}_{\Delta},\mathcal{O}_{S\times S})\simeq H^{0}(S\times S,\det\mathcal{N}_{\Delta/S\times S})$
to the constant global section $1$ of $\det\mathcal{N}_{\Delta/S\times S}\simeq\mathcal{O}_{S}$.
We let $\rho_{\Delta}:\mathcal{V}_{\Delta}\rightarrow T$ be the $\mathbb{G}_{a}$-bundle
over $T$ corresponding to this locally free resolution by the construction
of $\S$\ref{subsec:Bundle-Const} and we let $h=\pi\circ\rho_{\Delta}:\mathcal{V}_{\Delta}\rightarrow S$
\[\xymatrix{ \mathcal{V}_{\Delta} \ar[d]_{\rho_{\Delta}} \ar[dr]^{h} & \\ T \ar[r]^{\pi} & S.}\]  
\begin{prop}
\label{prop:V-Delta}With the notation above, the following hold: 

a) $\mathcal{V}_{\Delta}$ is a smooth affine variety of dimension
$5$.

b) The morphism $h=\pi\circ\rho_{\Delta}:\mathcal{V}_{\Delta}\rightarrow S$
is smooth and surjective.

c) For every closed point $j:\{p\}\hookrightarrow S$, we have commutative
diagram with cartesian squares \[\xymatrix{ V_p(S) \ar[r] \ar[d]_{\rho_p} & V_{\Delta} \ar[d]^{\rho_{\Delta}} \\ S\setminus\!\{p\} \ar[r] \ar[d]_{\rho_p} & T \ar[d]^{\pi} \\ \{p\} \ar[r]^{j} & S.}\]
\end{prop}
\begin{proof}
Since $\mathcal{V}_{\Delta}$ is affine over $S\times S$ by Proposition
\ref{prop:Bundle-construction} and $S\times S$ is affine, $\mathcal{V}_{\Delta}$
is an affine variety. The smoothness and the surjectivity of $h$
follow from those of $\rho_{\Delta}:\mathcal{V}_{\Delta}\rightarrow T$
and $\pi:T\rightarrow S$. By construction, the scheme theoretic fiber
$T_{p}$ of $\pi$ over a closed point $p\in S$ is isomorphic to
$S\setminus\{p\}$ while the restriction of the exact sequence \ref{eq:Second-Seq}
to the fiber $(S\times S)_{p}\simeq S$ of $\mathrm{pr}_{1}:S\times S\rightarrow S$
over $p$ is a free resolution \[\xymatrix{0 \ar[r] & \mathcal{O}_S \ar[r] & \mathcal{O}_S^{\oplus 2} \ar[r] & \mathcal{I}_p \ar[r] & 0}\] of
the ideal sheaf $\mathcal{I}_{p}\subset\mathcal{O}_{S}$ of $p$.
This implies in turn that $\mathcal{V}_{\Delta,p}\simeq V_{p}(S)$. 
\end{proof}
\begin{rem}
When $S$ is different from $\mathbb{A}^{2}$, it seems to be unknown
whether the middle locally free sheaf $\mathcal{F}$ occurring in
the exact sequence \ref{eq:Second-Seq} is free or not. If not free,
then $\mathcal{F}$ would provide a counter-example to the Generalized
Serre Problem on projective modules on topologically contractible
affine varieties, and $S\times S$ would not be contractible in the
unstable $\mathbb{A}^{1}$-homotopy category of Morel-Voevodsky. 
\end{rem}
For $S=\mathbb{A}^{2}=\mathrm{Spec}(\mathbb{C}[x_{1},y_{1}])$, the
ideal $I_{\Delta}$ of the diagonal in $\mathbb{A}^{2}\times\mathbb{A}^{2}=\mathrm{Spec}(\mathbb{C}[x_{1},y_{1}][x_{2},y_{2}])$
is generated by the polynomials $x_{2}-x_{1}$ and $y_{2}-y_{1}$,
and the corresponding variety $\rho_{\Delta}:\mathcal{V}_{\Delta}\rightarrow\mathbb{A}^{2}\times\mathbb{A}^{2}\setminus\Delta$
is thus isomorphic to the subvariety $\mathcal{W}$ of $\mathbb{A}^{2}\times\mathbb{A}^{2}\times\mathrm{Spec}(\mathbb{C}[U,V])$
defined by the equation $(x_{2}-x_{1})V-(y_{2}-y_{1})U=1$, equipped
with the restriction of the projection onto $\mathbb{A}^{2}\times\mathbb{A}^{2}$.
It follows in turn that the family $h:\mathcal{V}_{\Delta}\rightarrow\mathbb{A}^{2}$
is isomorphic to the trivial one $\mathbb{A}^{2}\times\mathrm{SL}_{2}(\mathbb{C})$
via the morphism 
\[
\Phi:\mathcal{W}\rightarrow\mathbb{A}^{2}\times\mathrm{SL}_{2}(\mathbb{C}),\,(x_{1},y_{1},x_{2},y_{2},U,V)\mapsto((x_{1},y_{1}),(x_{2}-x_{1}),(y_{2}-y_{1}),U,V).
\]
In contrast, when $S$ has non-negative Kodaira dimension, it follows
from Corollary \ref{cor:Contract-Iso-Crit} that the fibers $V_{p}(S)$
and $V_{p'}(S)$ of $h:\mathcal{V}_{\Delta}\rightarrow S$ over two
closed point $p$ and $p'$ of $S$ are isomorphic if and only if
$p$ and $p'$ belong to the same orbit of the action of $\mathrm{Aut}(S)$
on $S$. Since there does not exist any smooth contractible surface
of Kodaira dimension $0$, we have $\kappa(S)\geq1$ necessarily,
and since by Proposition \ref{prop:Auto-Contract} the action of $\mathrm{Aut}(S)$
on $S$ is not transitive, $h:\mathcal{V}_{\Delta}\rightarrow S$
is not isomorphic to the trivial family $S\times\mathrm{SL}_{2}(\mathbb{C})$.
In fact, as a consequence of Proposition \ref{prop:V-Delta} c) and
Corollary \ref{cor:Exotic-Sphere}, we have the following stronger
result:
\begin{thm}
Let $S$ be a smooth contractible surface of non-negative Kodaira
dimension with trivial automorphism group. Then $h:\mathcal{V}_{\Delta}\rightarrow S$
is a smooth family of pairwise non isomorphic exotic affine $3$-spheres. 
\end{thm}
\bibliographystyle{amsplain}

\end{document}